\documentclass[11pt,a4paper, reqno]{amsart}
\usepackage{amssymb,amsmath,amsthm,amstext,amsfonts}
\usepackage[dvips]{graphicx}
\usepackage{psfrag}
\usepackage{color}
\usepackage{xcolor}
\usepackage[colorlinks=true,linkcolor=blue, urlcolor=red, citecolor=blue]{hyperref}	
\usepackage{amsfonts}
\usepackage{epsfig}
\usepackage{graphicx}
\usepackage{amsmath}
\usepackage{amssymb}
\usepackage{color}
\usepackage{mathrsfs}
\usepackage{gfsartemisia}
\usepackage{cleveref}

\usepackage{soul}
\usepackage[Symbol]{upgreek}
\usepackage{xparse}
\usepackage{xcolor}
\usepackage{textgreek}
\definecolor{britishracinggreen}{rgb}{0.0, 0.26, 0.15}
\usepackage[T1]{fontenc}
\usepackage[latin1]{inputenc}
\usepackage[normalem]{ulem}
\newtheorem{theorem}{Theorem}[section]
\newtheorem{proposition}[theorem]{Proposition}
\newtheorem{lemma}[theorem]{Lemma}
\newtheorem{corollary}[theorem]{Corollary}

\newtheorem{remark}{Remark}[section]

\newtheorem{maintheorem}{Theorem}

\newcommand{\R}{\mathbb{R}}
\newcommand{\N}{\mathbb{N}}

\newcommand{\w}{\omega}

\DeclareMathAlphabet{\mathpzc}{OT1}{pzc}{m}{it}

\newcommand{\comment}[1]{}
\title[Genericity of trivial Lyapunov spectrum for $L^p$-cocycles derived from second order linear homogeneous differential equations]
{Genericity of trivial Lyapunov spectrum for $L^p$-cocycles derived from second order linear homogeneous differential equations}

\author[D. Amaro]{Dinis Amaro}
\thanks{$^\dagger$Corresponding author: bessa@ubi.pt.\\\emph{\textup{2020} Mathematics Subject Classification:} Primary: 34D08,  37H15,
	Secondary: 34A30, 37A20.}
\email{dinis.amaro@ubi.pt}

\author[M. Bessa]{M\'{a}rio Bessa$^\dagger$}
\email{bessa@ubi.pt}

\author[H. Vilarinho]{Helder Vilarinho\\ \\
\noindent \tiny{Centro de Matem\'atica e Aplica\c{c}\~oes (CMA-UBI)\\ Universidade da Beira
	Interior\\ Rua Marqu\^es d'\'Avila e Bolama, 6201-001, Covilh\~a, Portugal.}\\}
\email{helder@ubi.pt}


\date{\today}

\hypersetup{pdfauthor={Name}}
\usepackage{soul}

\begin{document}
	
	\begin{abstract}
	Given an ergodic flow $\varphi^t\colon M\rightarrow M$ defined on a probability space $M$ we study a family of continuous-time \emph{kinetic} linear cocycles associated to the solutions of the second order linear homogeneous differential equations $\ddot x +\alpha(\varphi^t(\omega))\dot x+\beta(\varphi^t(\omega))x=0$, where the parameters $\alpha,\beta$ evolve along the $\varphi^t$-orbit of $\omega\in M$. Our main result states that for a generic subset of kinetic continuous-time linear cocycles, where \emph{generic} means a Baire second category with respect to an $L^p$-like topology on the infinitesimal generator, the Lyapunov spectrum is trivial.\\

	\end{abstract}
	
	\maketitle

	\section{Introduction}
	\par We know, since the first half of nineteenth-century and by Liouville's theorem, that serious restrictions are present when we try to apply analytic methods to integrate most functions. This result can be seen as a kind of \emph{differential Galois theory} and represent a deep obstacle in solving differential equations explicitly. The way we bypass this inevitable fact is twofold: in one hand powerful numerical methods were developed to approximate the solutions and on the other hand a qualitative theory of differential equations emerged from the pioneering works of Poincar\'e and Lyapunov. We will be interested in following the last mentioned approach.
	
	With the study that we carry out we intend to understand the asymptotic behavior of the solutions of the second order homogeneous linear differential equations with coefficients displaying $L^p$ regularity, varying in time and allowing an $L^p$-small perturbation. Namely, to describe its Lyapunov spectrum under $L^p$-generic conditions of its coefficients. Families of this type of equations will be considered indexed in a flow which keep invariant the probability measure in a measure space $M$. We have as motivation in the first instance a family of equations that describe the motion of the simple damped pendulum  free from external forces, of type
	\begin{equation}\label{dampedp}
		\ddot{x}(t) +\alpha(\varphi^t(\omega))\dot{x}(t)+\beta(\varphi^t(\omega))x(t)=0,
	\end{equation}
	where $\alpha$ and $\beta$ are functions depending on $\omega\in M$ evolving along a flow $\varphi^t\colon M\rightarrow M$ for $t\in\mathbb{R}$. Clearly, if $\alpha$ and $\beta$ are first integrals related with $\varphi^t$ (i.e. constant along the $\varphi^t$-orbits) the equation (\ref{dampedp}) is easily solved by elementary methods of a first course on differential equations. When the parameters vary in time, explicit solutions are hard to get. This is the case when the \emph{frictional force} $\alpha$ and the \emph{frequency of the oscillator} $\beta$ change over time.
	
	In \cite{Be2} the second author deal with a similar case but with periodic coefficients along periodic closed orbits and proved that small $C^0$-perturbations on the parameters allows us to obtain that asymptotic unstable solutions are precisely the uniformly hyperbolic saddle-type ones. In the present paper we intend to focus on a perturbative theory with a coarser topology, namely allowing perturbations in an $L^p$-type topology and with random (non-periodic) base dynamics.
	
	Differential equations like (\ref{dampedp}) are ubiquitous in physics, engineering, biology and numerous applications of mathematics like, e.g., solid state physics, structural stability, wave propagation in one-dimensions, stability of synchronous electrical machines, etc. Fixing \emph{position} and \emph{momentum} $(x(0),\dot x(0))$ we intend to study the asymptotic behavior when $t\rightarrow\infty$ of the pair $(x(t),\dot x(t))$, namely the asymptotic exponential growth rate given by the \emph{Lyapunov exponent}. The literature on the subject with more or less similar nuances is substantial (see \cite{ACE, AW, Be2, FL, Lei} and the references therein). The literature on the broader matter of Lyapunov exponents had a substantial grown in the last decade as it can be seen by several books published lately \cite{Ba,DK,Iz,Pi,V}.
	
	The $L^p$-generic point of view on quite general linear differential systems was studied in \cite{BVi} by two of the authors and after Arbieto-Bochi \cite{AB} (see also the references therein). In \cite{BVi} was proved that the class of \emph{accesible} (aka \emph{twisting}) linear differential systems, a wider class that includes cocycles that evolve in $\text{GL}(d,\mathbb{R}), \text{SL}(d,\mathbb{R})$ and $\text{Sp}(d,\mathbb{R})$, have a trivial Lyapunov spectrum $L^{p}$-genericaly. When considering the sharper $C^0$-norm it is known since Millionshchikov's work in the late sixties that the generic behaviour changes (see \cite{Mi}). A complete treatment on the $C^0$-case was done in \cite{Be, Be2} after the discrete approach done in \cite{B,BV2}. The question of knowing the $C^0$-generic asymptotic behaviour of linear differential systems arising from equations like (\ref{dampedp}) is a work in progress. There as been recently a growing interest in understanding the discrete `dynamical cocycle' from an $L^p$-perturbative viewpoint and $L^p$-generic properties \cite{AABT2, FHT0, FHT}. Since perturbing the cocycle given by the derivative depends on a perturbation of the map, these dynamical cocycles tend to drag several other difficulties.
	
	When we intend to change all the Lyapunov exponents in order them to become equal the naive idea is to distribute expansion/contraction rates equally for all directions. This can be made by rotating directions in a convenient way to have, at the end of the day, those tax rates identically scattered. Indeed, rotating in a systematic way, is a crucial idea which was developed in certain contexts in the Sovietic literature of the 1970s (see \cite{No}) and in the 1980s by Ma\~n\'e \cite{M1,M2}. Issues related with the continuity of Lyapunov exponents are much more complicated within feeble topologies like the $L^p$ one. Here, we dedicate a substantial effort understanding the $L^p$-continuous dependence on the Lyapunov exponents as our arguments are supported on this assumption. Moreover, and also related with continuity, we know that typically squeezing the distance between vector fields implies in a decrease of the distance between its flows trajectories on compact times. Here, and once again, the $L^p$ topology creates additional difficulties.
	
	In overall, the main result in the present paper (Theorem~\ref{ops2}) can be summarized in the following way: 
	
\smallskip
	{\small\emph{For an $L^p$-generic choice of a kinetic linear differential system (as in \eqref{dampedp}) and for almost every driving realization, no matter what position and momentum $(x(0),\dot x(0))$ we chose as initial conditions, the asymptotic exponential behaviour of the solutions will be the same.}}
	
	\smallskip
	
	This paper is organized as follows: in \S\,\,\,\ref{continuous results} we will present the basic definitions and state our main result (Theorem~\ref{ops2}); section \S\,\,\,\ref{pmt} is devoted to the perturbation framework; in section \S\,\,\,\ref{usc} we will deal with continuity issues of the Lyapunov exponents with respecto to an $L^p$ distance and, finally, in \S\,\,\,\ref{PT1} we prove Theorem~\ref{ops2}.

	\section{Definitions and main result}\label{continuous results}
	
	\subsection{Linear cocycles} In this section we present some definitions that will be useful for the development of this work. Let $(M,\mathcal{M},\mu)$ be a probability space and let $ \varphi\colon \R \times M \to M$ be a \emph{metric dynamical system} (or \emph{flow}) in the sense that  is a measurable map and
	\begin{enumerate}
		\item $\varphi^t \colon M \to M$ given by $\varphi^t (\w) = \varphi(t,\w)$ preserves the measure $\mu$ for all $t \in \R$;
		
		\item $\varphi^0 = \text{Id}_{M}$ and $\varphi^{t+s}=\varphi^t\circ\varphi^s$ for all $t,s \in\R$.
	\end{enumerate}
	 Unless stated otherwise we will consider in the sequel that the flow is ergodic in the usual sense that there exist no invariant sets except null sets and
their complements.

Let $\mathcal{B}(X)$ be the Borel $\sigma$-algebra of a topological space $X$. A (continuous-time) linear \emph{random dynamical system} (RDS) on $(\R^2,\mathcal B(\R^2))$, or a (continuous-time) \emph{linear cocycle}, over $\varphi$ is a $(\mathcal{B}(\R)\times \mathcal{M}/\mathcal{B}(\text{GL}(2, \mathbb{R}))$-measurable map 	
	\begin{equation*}
	\Phi:\R\times M \to \text{GL}(2, \mathbb{R})
	\end{equation*}
	such that the mappings $\Phi(t,\w)$ forms a cocycle over $\varphi$, i.e.,
		\begin{enumerate}
			\item $\Phi(0,\w)=\text{Id}$ for all $\w\in M$; 
			\item $\Phi(t+s,\w)=\Phi(t,{\varphi^{s}(\w)})\circ\Phi(s,\w)$, for all $s, t\in\R$ and $\w\in M$,
		\end{enumerate}
	and $t\mapsto \Phi(t,\w)$ is continuous for all $\w\in M$. We recall that having $\w\mapsto\Phi(t,\w)$ measurable for each $t\in\R$ and $t\mapsto \Phi(t,\w)$ continuous for all $\w\in M$ implies that $\Phi$ is measurable in the product measure space. These objects are also called \emph{linear differential systems} in the literature (see \S\,\,\,\ref{kinetic}).

	\subsection{Kinetic Linear Differential Systems}\label{kinetic}
	
	We begin by considering as motivation the non-autonomous linear differential equation, which describes a motion of the damped harmonic oscillator as the `simple pendulum' along the path $(\varphi^{t}(\omega))_{t\in\mathbb{R}}$, with $\w\in M$ described by the flow $\varphi$.
		Let $K\subset \mathbb{R}^{2\times2}$ be the set of matrices $2\times 2$ of type
		\begin{eqnarray*}
		A=\left(\begin{matrix} 0&1\\b&a\end{matrix}\right)
	\end{eqnarray*}
for real numbers $a, b$, and denote by $\mathcal{G}$ the set of measurable applications $A: M\rightarrow \mathbb{R}^{2\times2}$  and by $\mathcal{K}\subset \mathcal{G}$ the set of \emph {kinetic} measurable applications $A: M\rightarrow K$. As usual we identify two applications on $\mathcal G$ that coincide on a $\mu$ full measure subset of $M$. Consider measurable maps $\alpha\colon M \to \R$ and $\beta\colon M \to \R$.
		Take the random differential equation given by \eqref{dampedp}  and defined by:
	\begin{equation*}
		\ddot{x}(t) + \alpha(\varphi^{t}(\omega)) \dot{x}(t) + \beta(\varphi^{t}(\omega)) x(t) = 0.
	\end{equation*}
		Considering the change of variables $y(t)= \dot{x}(t)$ we may rewrite \eqref{dampedp} as the following vectorial linear system
	\begin{eqnarray}\label{E1}
		\dot{X}= A(\varphi^{t}(\omega))\cdot X,
	\end{eqnarray}
	where $X=X(t)=(x(t),y(t))^T=(x(t),\dot x(t))^T$ and $A\in\mathcal{K}$ is given by
	\begin{eqnarray*}
		A(\omega)=\left(\begin{matrix}0&1\\ -\beta(\omega)& -\alpha(\omega)\end{matrix}\right).
	\end{eqnarray*}
		It follows from \cite[Thm. 2.2.2]{A} (see also Lemma 2.2.5 and Example 2.2.8 in this reference) that if $A\in \mathcal{G}^1=: \mathcal{G}\cap L^1(\mu)$, i.e. $\int_M \|A\|\,d\mu<\infty$, generates a unique (up to indistinguishability)  linear RDS $\Phi_A$ satisfying
	\begin{equation}\label{eq:LDS}
		\Phi_{A}(t,\omega)=\text{Id}+\int_{0}^{t}A(\varphi^{s}(\omega))\Phi_{A}(s,\omega)\,ds.
	\end{equation}

		The solution $\Phi_A(t,\w)$ defined in (\ref{eq:LDS}) is called \emph{the Carath\'eodory solution} or \emph{weak solution}. Given an initial condition $X(0)=v\in\mathbb{R}^2$, we say that $t\mapsto \Phi_A(t,\w)v$ solves or is a solution of the Random Differential Equation (RDE)~\eqref{E1}, or that the RDE~\eqref{E1} generates $\Phi_{A}(t,\w)$. Note  that $\Phi_A(0,\w)v=v$ for all $\w\in M$ and $v\in\mathbb{R}^{2}$. If the solution \eqref{eq:LDS} is differentiable in time (i.e. with respect to $t$) and satisfies for all $t$
	\begin{equation}\label{eq:LDS2}
		\frac{d}{dt}\Phi_A(t,\w)v=A(\varphi^t(\w))\Phi_A(t,\w)v\,\,\,\,\,\,\,\,\,\,\text{and}\,\,\,\,\,\,\,\,\,\, \Phi_A(0,\w) v=v,
	\end{equation}
	then it is called a \emph{classical solution} of RDE~\eqref{E1}.
	
Of course that $t\mapsto \Phi_A(t,\w)v\,$ is continuous for all $\w$ and $v$. Due to \eqref{eq:LDS2} we call $A:M\to K$ a (kinetic) `infinitesimal generator' of $\Phi_A$. Sometimes, due to the relation between $A$ and $\Phi_A$, we refer to both $A$ and $\Phi_A$ as a kinetic linear cocyle/RDS/differential system. If the RDE \eqref{E1} has initial condition $X(0)=v$ then $\Phi_A(0,\w)v=v$ and $X(t)=\Phi_A(t,\w)v$.

		\bigskip

\subsection{The $L^p$ topology}\label{top}
	We begin by defining an $L^p$-like topology generated by a metric that compares the infinitesimal generators on $\mathcal G$.By a standard use of Gr\"onwall's inequality arguments it is usual to obtain the approximation of flows (on compact times) by assuming the approximation of the corresponding infinitesimal generators. However, $L^p$-estimates, $1\leq p <\infty$, are more demanding and it is not clear if such $L^p$-continuous dependence holds (cf. Lemma~\ref{contdep2}). Related to this problem see e.g. \cite{Felix} where these issues are treated for $L^1$-approximation on initial conditions in quasilinear elliptic-parabolic equations. With this in mind we define now the distance we are going to consider along this work. Given $1\leq p < \infty$ and $A,B\in\mathcal{G}$ we 
	set 
			\begin{equation*}
		\hat\sigma_p(A,B):=\left\{\begin{array}{lll} \displaystyle\left(\int_{M} \|A(\omega)-B(\omega)\|^p\,d\mu(\w)\right)^{\frac1p}, \\    \infty\,\,\,\text{if the above integral does not exists,}         \\    \end{array}\right.
	\end{equation*}
and define
	\begin{equation*}
		\sigma_p(A,B):=\left\{\begin{array}{lll}\frac{\hat\sigma_p(\Phi_A,\Phi_B)}{1+\hat\sigma_p(\Phi_A,\Phi_B)},\,&&\text{if}\,\, \hat\sigma_p(A,B)<\infty\\   1,\,&&\text{if}\,\, \hat\sigma_p(A,B)=\infty         \\    \end{array}\right..
	\end{equation*}
	Clearly, $\sigma_p$ is a distance in $\mathcal{G}$.

\begin{proposition}\label{complete}
Consider $1\leq p <\infty$. Then:
\begin{enumerate}
\item[(i)] $\sigma_p(A,B)\leq \sigma_q(A,B)$ for all $q\geq p$ and all $A,B\in\mathcal{G}$. 
\item[(ii)] If $A\in\mathcal{G}^1$ then for any $B\in\mathcal{G}$ satisfying $\sigma_p(A,B)<1$ we have $B\in\mathcal{G}^1$. Therefore, $\sup\limits_{0\leq t\leq 1}\log^{+}\|\Phi_B(t,\omega)^{\pm1}\|\in L^{1}(\mu)$.
\end{enumerate}
\end{proposition}	

\begin{proof}

(i) For all $1\leq p\leq q <\infty$ it is a standard result on Lebesgue spaces that $L^q$-norms are thinner than $L^p$ ones. So, we have $\hat\sigma_p(A,B)\leq \hat\sigma_q(A,B)$. Hence, $\sigma_p(A,B)\leq \sigma_q(A,B)$.\\
(ii) Pick any $1\leq p <\infty$ and let $A\in\mathcal{G}^1$ and $B\in\mathcal{G}$ satisfying $\sigma_p(A,B)<1$. By (i) we have $\sigma_1(A,B)<1$, thus $\hat\sigma_1(A,B)<\infty$, and we also have that
$\|B(\w)\|\leq \|A(\w)\|+\|B(\w)-A(\w)\|$  for all $\w\in M$. Therefore, $B\in L^1(\mu)$. The last statement follows from~\eqref{eq IC}. 
\end{proof}

The following result will be used to prove that the kinetic subspace of linear cocycles is a Baire space.

\begin{lemma}
		$(\mathcal{K}^1,\sigma_p)$ is closed, for all $1\leq p <\infty$.
	\end{lemma}

    \begin{proof}
Given any sequence $(A_{n})_n$ in $\mathcal{K}^1$, with $A_{n}\stackrel{\sigma_p}{\longrightarrow}A^{*}$, for some $A^{*}\in \mathcal{G}$ we must prove that $A^{*}\in\mathcal{K}^1$. As $(A_{n})_n\in \mathcal{K}^1$ we get 
\[
A_n(\w)=\begin{pmatrix} 0 & 1 \\ -\beta_n(\w) & -\alpha_n(\w)\end{pmatrix}
\,\,\,\text{ and }\,\,\, A^{*}=\begin{pmatrix} a(\w) & b(\w) \\ -\beta(\w) & -\alpha(\w)\end{pmatrix}
\] for maps $\alpha_n, \beta_n\in L^1(\mu)$ and measurable maps $a$, $b$, $\alpha$ and $\beta$. From (i) in Proposition~\ref{complete} we have $A_{n}\stackrel{\sigma_1}{\longrightarrow}A^{*}$. From (ii) in Proposition~\ref{complete} we have that $a$, $b$, $\alpha$ and $\beta$ are also $L^1$ maps, i.e. $A^{*}\in \mathcal{G}^1$. Hence, $\int_M |a(\w)|\,d\mu(\w)=0$ and $\int_M |b(\w)-1|\,d\mu(\w)=0$, which implies $a(\w)=0$ and $b(\w)=1$ for $\mu$ almost every $\w$. Therefore, $A^{*}\in \mathcal{K}$. In conclusion, $A^{*}\in  \mathcal{G}^1\cap \mathcal{K}=\mathcal{K}^1$.
     \end{proof}

Since for all $1\leq p <\infty$, the metric space $(\mathcal{G}^1,\sigma_p)$ is complete and $\mathcal{K}^1\subset\mathcal{G}^1$ is $\sigma_p$-closed, we conclude the following.

\begin{corollary}\label{complete 2}
		For all $1\leq p <\infty$, $(\mathcal{K}^1,\sigma_p)$ is a complete metric space and, therefore, a Baire space.
	\end{corollary}

	\subsection{Lyapunov exponents}
	Let $\mathcal{K}^1=\mathcal{K}\cap L^1(\mu)\subset \mathcal{G}^1$. Observe that if $A\in\mathcal{K}^1$ then the cocycle $\Phi_A$ satisfies the following \emph{integrability condition}
	\begin{equation}\label{eq IC}
	\sup\limits_{0\leq t\leq 1}\log^{+}\|\Phi_A(t,\omega)^{\pm1}\|\in L^{1}(\mu),
	\end{equation}	
	where $f^+=\max\{0,f\}$. Indeed, consider $\w$ in the full measure $\varphi^t$-invariant subset of $M$ where $t\mapsto A(\varphi^t(\omega))$ is locally integrable. By \eqref{eq:LDS}  we get $$\|\Phi_{A}(t,\omega)^{\pm1}\|\leq1+\int_{0}^{t}\left\|A(\varphi^{s}(\omega))\right\|\,\left\|\Phi_{A}(s,\omega)^{\pm1}\right\|\,ds.$$ By Gr\"onwall's inequality (see \cite{A}) we have 
	\begin{equation}\label{Gronwall}
	\Phi_{A}(t,\omega)^{\pm1}\|\leq \exp\left(\int_{0}^{t}\|A(\varphi^{s}(\omega))\|\,ds\right)
	\end{equation} 
	that is, for all $t\geq 0$ we have $\log^+\|\Phi_{A}(t,\omega)^{\pm1}\|\leq \int_{0}^{t}\|A(\varphi^{s}(\omega))\|\,ds$. Therefore 
	\begin{equation}\label{eqIC3}
	\sup\limits_{0\leq t\leq T}\log^+\|\Phi_{A}(t,\omega)^{\pm1}\|\leq\int_{0}^{T}\|A(\varphi^{s}(\omega))\|\,ds=:\psi(\w,T).
	\end{equation}
	 By \cite[Lemma 2.2.5]{A} we have $\psi(\cdot,T)\in L^1(\mu)$. Thus $\sup\limits_{0\leq t\leq 1}\log^+\|\Phi_A(t,\w)^{\pm1}\|\in L^1(\mu)$, getting the claim. 	 
	 Moreover, Fubini's theorem allow us to obtain
	\begin{align*}
	\int_M\log^+\|\Phi_{A}(t,\omega)^{\pm1}\|\,d\mu(\omega)&\leq \int_M\int_{0}^{t}\|A(\varphi^{s}(\omega))\|\,ds\,d\mu(\omega)\\
	&=\int_{0}^{t}\int_M\|A(\varphi^{s}(\omega))\|\,d\mu(\omega)\,ds=t\|A\|_1.
	\end{align*}	
	
	If $A\in{\mathcal{G}^1}$ then the integrability condition \eqref{eq IC} holds and Oseledets theorem (see e.g. \cite{O,A}) guarantees that for $\mu$ almost every $\omega\in M$, there exists a $\Phi_A$-invariant splitting called \emph{Oseledets splitting} of the fiber $\mathbb{R}^{2}_{\omega}=E^{1}_{\omega}\oplus E^{2}_{\omega}$ and real numbers called \emph{Lyapunov exponents} $\lambda_{1}(A,\omega)\geq\lambda_{2}(A,\omega)$, such that:
	\begin{equation*}
	\underset{t\rightarrow{\pm{\infty}}}{\lim}\frac{1}{t}\log{\|\Phi_A(t,\omega) v^{i}\|=\lambda_{i}(A,\omega)},
	\end{equation*}
	for any $v^{i}\in{E^{i}_{\w}\setminus\{\vec{0}\}}$ and $i=1, 2$. Moreover, given any of these subspaces $E^{1}_{\w}$ and $E^{2}_{\w}$, the angle between them along the orbit has subexponential growth, meaning that
	\begin{equation}\label{angle}
	\lim_{t\rightarrow{\pm{\infty}}}\frac{1}{t}\log\sin\left(\measuredangle(E^{1}_{\varphi^{t}(\omega)},E^{2}_{\varphi^{t}(\omega)})\right)=0.
	\end{equation}
		If the flow $\varphi^{t}$ is ergodic, then the Lyapunov exponents and the dimensions of the associated subbundles are constant $\mu$ almost everywhere and we refer to them as $\lambda_1(A)$ and as $\lambda_2(A)$, with $\lambda_1(A)\geq\lambda_2(A)$. We say that $A$ has \emph{one-point Lyapunov spectrum} or \emph{trivial Lyapunov spectrum} if for $\mu$ a.e. $\w\in M$, $\lambda_{1}(A,\omega)= \lambda_{2}(A,\omega)$. Otherwise we say  $A$ has \emph{simple Lyapunov spectrum}. For details on these results on linear differential systems see \cite{A} (in particular, Example 3.4.15). See also ~\cite{JPS}.
				
	\medskip

	\subsection{Statement of the main result}
	We are now in position to state our main contribution. In the present paper we establish the existence of a $\sigma_p$-residual of $\mathcal{K}^1$ displaying one-point spectrum. More precisely we prove the following:
	
	\bigskip
	
	\begin{maintheorem}\label{ops2}
		For all $1\leq p < \infty$ there exists a $\sigma_p$-residual subset $\mathcal{R}\in \mathcal{K}^1$ such that any $A\in\mathcal{R}$ has one-point spectrum.
	\end{maintheorem}
	
	\bigskip

	\section{Interchanging Oseledets' directions}\label{pmt}
	In this section we build up the fundamental perturbation tool which allows us to interchange Oseledets directions. First, we will discuss some topics about the perturbations that we will be tailor-made to our purpose.
	
	\medskip
	
	\subsection{Perturbations supported in compact sets}
Take $A\in\mathcal{G}$ and a non-periodic orbit $\omega\in M$. We will consider a perturbation $B_{\w}$ of $A$ only along a segment of the orbit of $\omega$ with extremes $\w$ and $\varphi^{1}(\w)$. Let $P\in\mathcal{G}$ be given and define $B\colon M\to \mathbb{R}^{2\times2}$ such that $B(\hat\w)=A(\hat\w)$ for all $\hat\w$ outside $\varphi^{[0,1]}(\w)=\{\varphi^{s}(\w): s\in[0, 1]\}$ and $B(\hat\w)=P(\hat\w)$ otherwise. The map $B$ is called  a \emph{perturbation of $A$ by $P$ supported on $\varphi^{[0,1]}(\w)$}.

\begin{lemma}\label{rot4cont}
Given $\w\in M$, $u,v\in\mathbb{R}^2\setminus\{ 0\}$, $A\in\mathcal{K}^1$, there is $\upgamma\neq 0$, and a perturbation $B_\w\in\mathcal{K}^1$ of $A$ supported on $\varphi^{[0,1]}(\w)$ such that:
  \begin{enumerate}
  \item[(i)] $\|B_\w(\hat \w)\|\leq 4\pi^2$ for all $\hat\w$ on $\varphi^{[0,1]}(\w)$ and
  \item[(ii)] $\Phi_{B_\w}(1,\omega)u=\upgamma\, v$.
  \end{enumerate}

\end{lemma}
	
\begin{proof}
Let $\theta=\measuredangle(\mathbb{R}u,\mathbb{R}v)\in [\pi,2\pi]$ measured clockwise. Set a constant infinitesimal generator $P\colon M\to \mathbb{R}^{2\times2}$ given by 
\begin{equation}\label{infgen}
P=\left(\begin{matrix} 0 & 1\\ -\theta^2 & 0\end{matrix}\right).
\end{equation} 
We consider the perturbation $B_\w\in\mathcal{K}^1$ of $A$ by $P$ supported on $\varphi^{[0,1]}(\w)$. 
The infinitesimal generator in \eqref{infgen} generates a linear differential system with fundamental classical solution \eqref{eq:LDS2} given, for all $\hat\w\in M$ and $t\in\mathbb{R}$ by the `clockwise elliptical rotation' defined by:
\begin{equation*}
\Phi_{P}(t,\w)=\left(\begin{matrix}\cos(\theta t)& \theta^{-1}\sin(\theta t)\\ -\theta\sin(\theta t) & \cos(\theta t)\end{matrix}\right),
\end{equation*} 
and such that  $\Phi_{B_\w}(1,\w)u=\Phi_{P}(1,\w)u=\upgamma v$, for some $\upgamma\neq0$ fulfilling (ii). 
\end{proof}

\begin{remark}\label{maisfrente}
In the application of Lemma~\ref{rot4cont} further ahead we will consider $u$ and $v$ such that $\mathbb{R} u=E_{\w}^1$ and $\mathbb{R} v=E_{\varphi^{1}(\w)}^2$ where $E^i$ are the Oseledets directions associated to $A$. In particular,  $\Phi_{B_\w}(1,\w)E_{\w}^1=E_{\varphi^{1}(\w)}^2$. Morever, as the canonic norm and the maximum norm are equivalent item (i) of Lemma \ref{rot4cont} and~\eqref{Gronwall} impies $\|\Phi_{B_\w}(1,\w)\|_{\max}\leq C \|\Phi_{B_\w}(1,\w)\|\leq C \text{e}^{4\pi^2}$.
\end{remark}

\subsection{Special flows}
Consider a measure space $\Sigma$, a map $\mathcal{T}\colon \Sigma\rightarrow{\Sigma}$, a $\mathcal{T}$-invariant probability
measure $\tilde{\mu}$ defined in $\Sigma$ and a roof function
$h\colon \Sigma\rightarrow{\mathbb{R}^{+}}$ satisfying $h(\w)\geq H>0$, for some $H>0$ and all $\w\in{\Sigma}$, and 
 $
 \int_{\Sigma}h(\w)d\tilde{\mu}(\w)<\infty
$.

Define the space $M_{h}\subseteq{\Sigma\times{\mathbb{R}_+}}$ by
$$
M_h=\{(\w,t) \in \Sigma\times{\mathbb{R}_+}: 0 \leq t \leq h(\w) \}
$$
with the identification between the pairs $(\w,h(\w))$ and $(\mathcal{T}(\w),0)$. The semiflow defined
on $M_h$ by $S^s(\w,r)=(\mathcal{T}^{n}(\w),r+s-\sum_{i=0}^{n-1}h(\mathcal{T}^{i}(\w)))$,
where $n\in{\mathbb{N}}$ is uniquely defined by
$$
\sum_{i=0}^{n-1}h(\mathcal{T}^{i}(\w))\leq{r+s}<\sum_{i=0}^{n}h(\mathcal{T}^{i}(\w))
$$
is called a \emph{suspension semiflow}. If $\mathcal{T}$ is invertible then $(S^t)_t$ is a flow.
Furthermore, if $\ell$ denotes the one dimensional Lebesgue measure the measure $\mu=(\tilde\mu \times \ell)/\int h\, d\tilde\mu$ defined on $M_h$
by
$$
\int g \, d\mu= \frac{1}{\int h\, d\tilde\mu} \int \left( \int_0^{h(\w)} g(\w,t) dt \right)\, d\tilde\mu(\w),  \quad \forall g\in C^0(M_h)
$$
is a probability measure and it is invariant by the suspension semiflow $(S^t)_t$. Flows with such representation are called \emph{special flows}.

\medskip

Next result, despite simple, will be the key step to prove second part of Proposition~\ref{MProp}. We fix a special flow described by the quadruple $(\varphi^t,\Sigma,\mathcal{T},h)$ endowed with product measure $\tilde\mu\times\ell$, $\Sigma_0\subseteq \Sigma$ with $\tilde{\mu}(\Sigma_0)>0$. Moreover, for $b>a>0$ we define the set
\begin{equation*}
\varphi^{[a,b]}(\Sigma_0)=\{\varphi^t(\w)\colon \w\in \Sigma_0,\, t\in[a,b]\}.
\end{equation*}
Given $A\in\mathcal{G}^1$, $P\in\mathcal{G}$, $\Sigma_0\subseteq\Sigma$ and $a>0$, we may extend the perturbation of $A$ by $P$ to be supported on the flowbox $\varphi^{[a,a+1]}({\Sigma_0})$ in the following way: for $\w\in \varphi^{[a,a+1]}({\Sigma_0})$ we project $\w$ in $\tilde\w\in\Sigma_0$ and let $B_{\tilde\w}$ be perturbation of $A$ by $P$ supported on $\varphi^{[0,1]}(\varphi^a(\w))$, and define
\begin{equation}\label{perturbacao global}
		B(\w):=\left\{\begin{array}{lll}A(\w),\,&&\text{if}\,\, \w\notin  \varphi^{[a,a+1]}({\Sigma_0})\\   B_{\tilde \w}(\w),\,&&\text{if}\,\, \w\in  \varphi^{[a,a+1]}({\Sigma_0})         \\    \end{array}\right..
	\end{equation}
To distinguish the situations we refer for $B(\w)$ as a \emph{global perturbation of $A$ by $P$ supported in $\varphi^{[a,a+1]}(\Sigma_0)$}.

\begin{lemma}\label{rot4pert}
  For all $A\in\mathcal{G}^1$, $a>0$ and $\varepsilon\in]0,1[$, there exists a measurable set $\Sigma_0\subset \Sigma$ with $\tilde\mu(\Sigma_0)>0$ such that for any global perturbation $B\in\mathcal{G}^1$ of $A$ supported in the flowbox $\varphi^{[a,b]}(\Sigma_0)$ with $\|B(\varphi^t(\w))\|\leq 4\pi^2$ for all $\w\in \Sigma_0$ and $t\in[a,b]$, we have that $\sigma_1(A,B)<\varepsilon$.
\end{lemma}
\begin{proof}
Let $A\in\mathcal{G}^1$, $b>a>0$ and $\varepsilon\in]0,1[$ be fixed. Since $A\in\mathcal G^1$ we let $L>0$ be such that $\int_M\|A(\w)\|\,d\mu(\w)<L$. For any $\Sigma_0\subset \Sigma$ we have $\mu(\varphi^{[a,b]}(\Sigma_0))= (b-a)\tilde{\mu}(\Sigma_0)$. Let $\varepsilon'=\frac{\varepsilon}{1-\varepsilon}$ and choose a measurable set $\Sigma_0\subset \Sigma$ satisfying
\begin{equation}\label{sigma}
\tilde\mu(\Sigma_0)\in \left]0,\frac{\varepsilon'}{(b-a)(L+4\pi^2)}\right[.
\end{equation}
Finally, we will check that $\sigma_1(A,B)<\varepsilon$. Indeed, from \eqref{sigma} we get:
\begin{eqnarray*}
\hat\sigma_1(A,B)&=&\int_{M}\|A(\omega)-B(\omega)\|\,d\mu(\w)=\int_{\varphi^{[a,b]}(\Sigma_0)}\|A(\omega)-B(\omega)\|\,d\mu(\w)\\
&\leq& \int_{\varphi^{[a,b]}(\Sigma_0)} \|A(\w)\| + \|B(\omega)\|\,d\mu(\w)\leq \mu(\varphi^{[a,b]}(\Sigma_0))(L+4\pi^2)\\
&=&(b-a)(L+4\pi^2)\tilde{\mu}(\Sigma_0)<\varepsilon',
\end{eqnarray*}
which implies $\sigma_1(A,B)<\varepsilon$.
\end{proof}

\subsection{Lowering the maximal Lyapunov exponent} The next result asserts that is possible to lower the maximal Lyapunov exponent of a kinetic cocycle with simple spectrum by a $\sigma_p$-small perturbation.
			
\begin{proposition}\label{MProp}
  Given $A\in\mathcal{K}^1$ and $\varepsilon,\delta>0$, there exists $B\in\mathcal{K}^1$ such that $\sigma_1(A,B)<\varepsilon$ and 
 \begin{equation}\label{tstat}
  \lambda_1(B)\leq \frac{\lambda_1(A)+\lambda_2(A)}2+\delta.
  \end{equation}
\end{proposition}

\begin{proof}
Let $A\in\mathcal{K}^1$ and $\varepsilon,\delta>0$ be given. We assume that $A$ has simple spectrum, otherwise the conclusion is trivial. We are going to perform perturbations supported on a segment of size $1$ (cf. Lemma~\ref{rot4cont}) in points $\w$ of a subset $\Sigma_0$ of positive measure, with the perturbation to be taken approximately at $\varphi^{N_0/2}(\w)$, where $N_0$ is large enough to fulfill the asymptotic properties of Oseletets' theorem, up to some given accuracy $\eta$, both for $\w$ and $\varphi^{\frac{N_0}{2}+1}(\w)$. To avoid unpleasant overlapping situations, we codify $\varphi$ by a special flow along a Kakutani castle with a finite roof function. For a better understanding and since the flow $\varphi$ is ergodic a simple application of Rudolph's two symbol code representation (see \cite{Rudo}) allows us to consider a special flow $S^t$ with basis $\Sigma\subset M$ and a roof function $h\colon \Sigma\to \{N_0+1,N_0+q\}$, for some large $N_0$ and some $1<q\in \mathbb{R}\setminus \mathbb{Q}$. We recall that Rudolph's theorem gives a step function of irrationally independent high levels, and we need \emph{time} enough to \emph{see} the Lyapunov exponent in time $N_0/2$ (with accuracy $\eta$), to perform the perturbation in time less than $1$, and so \emph{see} again the Lyapunov exponent in time $N_0/2$ before the return to the basis. In this way we prevent overlaps of the perturbations.  Moreover, we have a decomposition of $\mu=\tilde\mu\times\ell$, where $\ell$ is the time length and $\tilde\mu$ is a measure on $\Sigma$ invariat by the return map. By abuse of notation in the following we still denote the special flow by $\varphi^t$.
Up to some rearrangement of the castle, given  $\eta>0$ there is $N\in\N$ ($N\geq N_0$) such that by the Oseledets theorem there is a subset $ \tilde{\Sigma}\subseteq \Sigma$, with $\tilde\mu(\tilde{\Sigma})>0$, such that for every $\w\in \tilde{\Sigma}$ we have:
     \begin{equation}\label{eta}
       \left|\lambda_i-\frac1t\log\left\|\Phi_A(t,\w)\vert_{E_\w^i}\right\|\right|<\eta/8<\eta
     \end{equation}
       for $i=1,2$ and $t\geq N/2$. Notice that from the cocycle property, by taking $N$ larger if necessary, for $\w\in \tilde{\Sigma}$ and setting $\w'=\varphi^{N/2+1}(\w)$, we also have 
      \begin{equation}\label{eta2}
        \left|\lambda_i-\frac1{N/2-1}\log\left\|\Phi_A(N/2-1,\w')\vert_{E_{\w'}^i}\right\|\right|<\eta.
      \end{equation}
Indeed, to obtain \eqref{eta2} we notice that
{\footnotesize\begin{align*}
\left|\lambda_i-\frac1{N/2-1}\log\left\|\Phi_A(N/2-1,\w')\vert_{E_{\w'}^i}\right\|\right| 
&= \left|\lambda_i-\frac1{N/2-1}\log\left\|\Phi_A(N,\w)\vert_{E_\w^i}\Phi_A(-(N/2+1),\w')\vert_{E_{\w'}^i}\right\|\right|
\\
 &\leq \left|2\lambda_i-\frac{N}{N/2-1}\frac1N\log\left\|\Phi_A(N,\w)\vert_{E_\w^i}\right\|\right| \\
&+  \left|-\lambda_i- \frac{N/2+1}{N/2-1}\frac{1}{N/2+1}\log\left\|\Phi_A(N/2+1,\w)^{-1}\vert_{E_{\w'}^i}\right\|\right|.
\end{align*}       }
The first term is less or equal than
     \begin{align*}
\left|2-\frac{N}{N/2-1}\right||\lambda_i|+\left|\lambda_i-\frac1N\log\|\Phi_A(N,\w)\vert_{E_\w^i}\right| \left|\frac{N}{N/2-1}\right|,
\end{align*}    
which can be made smaller than $\eta/2$ for sufficiently large $N$, as well the second term because is less or equal than
 \begin{align*}
\left|1-\frac{N/2+1}{N/2-1}\right||\lambda_i|+\left|\lambda_i-\frac{1}{N/2+1}\log\|\Phi_A(N/2+1,\w)\vert_{E_\w^i}\right| \left|\frac{N/2+1}{N/2-1}\right|.
\end{align*}
since
\[\left|-\lambda_i- \frac1{N/2-1}\log\|\Phi_A(N/2+1,\w)^{-1}\vert_{E_{\w'}^i}\|\right|=\left|-\lambda_i+ \frac1{N/2-1}\log\left\|\Phi_A(N/2+1,\w)\vert_{E_{\w'}^i}\right\|\right|.\]

For every  $\tilde\w\in \tilde{\Sigma}$ let $B_{\tilde\w}$ be the perturbation of $A$ by $P$, depending on $\mathbb{R}u=E_{\varphi^{N/2}(\tilde\w)}^1$ and $\mathbb{R}v=E_{\varphi^{1+N/2}(\tilde\w)}^2$, as in \eqref{infgen} given by Lemma~\ref{rot4cont} supported on $\varphi^{[N/2,N/2+1]}(\tilde\w)$  such that we have
    \[
    \Phi_{B_{\tilde\w}}(1,\varphi^{N/2}(\tilde\w))E_{\varphi^{N/2}(\tilde\w)}^1=E_{\varphi^{1+N/2}(\tilde\w)}^2.
     \]
     We may consider now a global perturbation $B$ of $A$ supported on the flowbox $\varphi^{[N/2,N/2+1]}(\tilde{\Sigma})$ as before: for $\w\in \varphi^{[N/2,N/2+1]}(\tilde{\Sigma})$ we project $\w$ in $\tilde\w\in\tilde\Sigma$ and define
\begin{equation*}
		B(\w):=\left\{\begin{array}{lll}A(\w),\,&&\text{if}\,\, \w\notin  \varphi^{[N/2,N/2+1]}(\tilde{\Sigma})\\   B_{\tilde \w}(\w),\,&&\text{if}\,\, \w\in  \varphi^{[N/2,N/2+1]}(\tilde{\Sigma})         \\    \end{array}\right..
	\end{equation*}
Notice that we have $B\in\mathcal{K}^1$. By passing to a subset of $\tilde{\Sigma}$ if necessary, from Lemma~\ref{rot4pert} we may assume that $\sigma_1(A,B)<\varepsilon$.

We will see that we may choose a large $N$ depending on $\eta$ (depending on $\delta$) such that for all $\w\in \tilde{\Sigma}$ we have
    \begin{equation}\label{eq:dec lyap exp}\lambda_1(B,\w)\leq \frac{\lambda_1(A,\w)+\lambda_2(A,\w)}2+\delta.\end{equation}
   Having $\tilde\mu(\tilde{\Sigma})>0$  implies that~\eqref{eq:dec lyap exp} holds for all $\w$ in a subset $\tilde M\subseteq M_h$ with $\mu(\tilde M)>0$, and since we are assuming that the flow $\varphi^t$ is $\mu$-ergodic, we have
     \[\lambda_1(B)\leq \frac{\lambda_1(A)+\lambda_2(A)}2+\delta\]
     as required in \eqref{tstat}.  Let us explain  the effect of mixing Oseledets' directions on the decay of the Lyapunov exponents.
Fiz $\eta>0$ to be choosed later and take $\omega\in \tilde{\Sigma}$ with Oseledets directions $E^1_\omega\oplus E^2_\omega$ and $N>0$ sufficiently large in order to have \eqref{eta} and \eqref{eta2} for $\omega^\prime=\varphi^{N/2+1}(\omega)$, that is
\begin{equation}\label{steps}
\Phi_A\left(N/2,\omega\right)\cdot v^1_\omega\approx \text{e}^{\lambda_1\frac{N}{2}},\Phi_A\left(N/2,\omega\right)\cdot v^2_\omega\approx \text{e}^{\lambda_2\frac{N}{2}},\\
 \end{equation}
 and
 \begin{equation}\label{steps2}
\Phi_A\left({N}/{2}-1,\omega^\prime\right)\cdot v^1_{\omega^\prime}\approx \text{e}^{\lambda_1(\frac{N}2-1)},\Phi_A\left(N/2-1,\omega^\prime\right)\cdot v^2_{\omega^\prime}\approx \text{e}^{\lambda_2(\frac{N}2-1)},
 \end{equation}
	where $\lambda_1=\lambda_1(A)$ and $\lambda_2=\lambda_2(A)$ are the Lyapunov exponents of $A$, $v_{\w}^i\in E^i_{\w}$ and $v_{\w'}^i\in E^i_{\w'}$ are unitary vectors, and $\approx$ means $\eta$-close as in \eqref{eta} and \eqref{eta2}.

Notice that, for all $\w\in \tilde{\Sigma}$ we have:
\begin{equation}\label{three}
\Phi_B(N,\omega)=\Phi_A\left(N/2-1,\omega'\right)\cdot\Phi_{B_\w}(1,\varphi^{N/2}(\w))\cdot \Phi_A\left({N}/{2},\omega\right).
\end{equation}
So, for all $\w\in \tilde{\Sigma}$, we may decompose the action of the map $\Phi_B({N},\omega)$ in three pieces:
\begin{itemize}
\item [\textbullet] The first, between $\omega$ and $\varphi^{N/2}(\omega)$, and the  third, between $\w'$ and $\varphi^{N}(\omega)$, that, using the basis induced by the Oseledets directions with respect to the splitting $E^1\oplus E^2$, may be writen as
	$$\Phi_A\left(N/2,\omega\right)=\begin{pmatrix}   \phi_{1,1}\left(N/2,\omega\right) & 0 \\0 & \phi_{2,2}\left(N/2,\omega\right)
\end{pmatrix}$$ \:\text{and}\: $$\Phi_A\left(N/2,\omega'\right)=\begin{pmatrix}\phi_{1,1}\left(N/2-1,\omega'\right) & 0 \\0 & \phi_{2,2}\left(N/2-1,\omega'\right)
\end{pmatrix}.$$
Using \eqref{steps} and \eqref{steps2} we get that $\Phi_A\left(N/2,\omega\right)$ is an operator that can be represented approximately (with $\left|\phi_{1,1}- \text{e}^{\frac{\lambda_1}{2}N}\right|<\eta$ and $\left|\phi_{2,2}- \text{e}^{\frac{\lambda_2}{2}N}\right|<\eta$) by the matrix
\begin{equation}\label{three1}
\begin{pmatrix}
\text{e}^{\frac{\lambda_1}{2}N} & 0 \\ 0 & \text{e}^{\frac{\lambda_2}{2}N}
\end{pmatrix}
\end{equation}
written, as usual, in the Oseledets basis $\{v_\w^1,v_\w^2\}$.
Similarly, $\Phi_A(N/2,\omega')$ can be represented approximately (with $\left|\phi_{1,1}- \text{e}^{\lambda_1(N/2-1)}\right|<\eta$ and $\left|\phi_{2,2}- \text{e}^{\lambda_2(N/2-1)}\right|<\eta$) by the matrix
\begin{equation}\label{three3}
\begin{pmatrix}\text{e}^{\lambda_1(N/2-1)} & 0 \\0 & \text{e}^{\lambda_2(N/2-1)}
\end{pmatrix}
\end{equation}
in the Oseledets basis $\left\{v_{\w'}^1,v_{\w'}^2\right\}$.

\item [\textbullet] The second piece, between $\varphi^{N/2}(\w)$ and $\omega^\prime$, with matrix in the basis induced by the Oseledets directions with respect to the splitting $E^1_{\varphi^{N/2}(\w)}\oplus E^2_{\varphi^{N/2}(\w)}$ which we denote by:

\begin{equation}\label{three2}
\Phi_{B_{\w}}(1,\varphi^{N/2}(\omega))=\begin{pmatrix} 0 & a_{1,2} \\ a_{2,1} & a_{2,2} \end{pmatrix}.\end{equation}

	\end{itemize}
	 With this interchange of directions the largest grow for $\Phi_A(N,\w)$ (given by $\text{e}^{\lambda_1 N}$) can never be achieved for the perturbation $\Phi_B(N,\w)$. Indeed, the entries of $\Phi_B(N,\w)$ are bounded by:
	 $$\max\left\{\text{e}^{\lambda_1(N/2-1)}\text{e}^{\frac{\lambda_2}{2}N},e^{\lambda_2(N/2-1)}\text{e}^{\frac{\lambda_1}{2}N}, \text{e}^{\frac{\lambda_2}{2}N}\text{e}^{\lambda_2(N/2-1)}\right\}\times\max\left\{a_{1,2}, a_{2,1}, a_{2,2} \right\}$$
	  We now estimate  $\frac{1}{N}\log\|\Phi_B(N,\w)\|$. By \eqref{three} we know that $\Phi_B(N,\w)$ is the composition of the matrices \eqref{three1}, \eqref{three2} and \eqref{three3}. For the sake of simplicity of presentation we estimate $\|\Phi_B(N,\w)\|$ using these three matrices which are in Oseledets basis. However, since the angle between the Oseledets fibers has subexponential decay as in~\eqref{angle}, estimating $\|\Phi_B(N,\w)\|$ by considering the canonical basis only increases the technicalities which we avoid in the sequel. Taking this into consideration and also Remark~\ref{maisfrente} we can ensure that:
	 $$\|\Phi_B(N,\w)\|\leq \max\left\{\text{e}^{\frac{\lambda_1+\lambda_2}{2}N-\lambda_1},\text{e}^{\frac{\lambda_1+\lambda_2}{2}N-\lambda_2}\right\}\,C\text{e}^{4\pi^2}.$$
	 Therefore,
	 $$\frac{1}{N}\log \|\Phi_B(N,\w)\|\leq \max\left\{\frac{\lambda_1+\lambda_2}{2}-\frac{\lambda_1}{N},\frac{\lambda_1+\lambda_2}{2}-\frac{\lambda_2}{N}\right\}+\frac{\log C\text{e}^{4\pi^2}}{N}.$$
Finally, taking $N$ sufficiently large we obtain
\begin{equation*}
 \frac{1}{N}\log\|\Phi_B(N,\omega)\|\leq\frac{\lambda_1(A,\w)+\lambda_2(A,\w)}{2}+\delta.
\end{equation*}
By (\ref{exp via inf}) we have
	
\[
\begin{split}
   \lambda_1(B,\w) & = \lim_{t\to\infty}\frac{1}{t}\log\|\Phi_B(t,\omega)\|   = \lim_{n\to\infty}\frac{1}{n}\log\|\Phi_B(n,\omega)\|  = \underset{n}{\inf}\,\frac{1}{n}\log\|\Phi_B(n,\omega)\| \\
   & \leq \frac{1}{N}\log\|\Phi_B(N,\omega)\|  \leq\frac{\lambda_1(A,\w)+\lambda_2(A,\w)}{2}+\delta
\end{split}
\]
and \eqref{eq:dec lyap exp} is proved.

\end{proof}

	\medskip
		
\section{Upper semi-continuity for the top Lyapunov exponent}\label{usc}

	\subsection{On $L^1$-continuous dependence results} \label{cdres}
In this subsection we provide some preliminary results to achieve the upper-semicontinuity for the top Lyapunov exponent. We start with a basic measure-theoretical result~that will be instrumental.

\begin{lemma}\label{funcion f}
Let $f\in L^{1}(\mu)$, $f\geq0$. For all	$\eta>0$ exists $K>0$ such that for all $h\in L^{1}(\mu)$ with $h\geq0$ and $\|h-f\|_{1}<\eta$, we have
\begin{equation*}
\int_{\{h>K\}}\!\!\!h\: d\mu(\w)<2\eta\qquad\text{ and }\qquad \mu(\{h>K\})<\frac{2\eta}{K}.
\end{equation*}
\end{lemma}
\begin{proof}\cite[Lemma 5]{AB}.
\end{proof}
Throghout this subsection we consider $A,B\in \mathcal{G}^1$, $\epsilon>0$ and $T\in\mathbb{N}$. Let us define
\begin{equation}\label{fg}
f(\w)=\int_0^1\|A(\varphi^s(\w))\|\,ds\quad\text{ and }\quad g(\w)=\int_0^1\|B(\varphi^s(\w))\|\,ds.
\end{equation}
As we already said, by \cite[Lemma 2.2.5]{A}  we have that $f,g \in L^1(\mu)$. We set $\eta=\varepsilon/(16(T+2))$ and fix $B$ satisfying  $\hat\sigma_p(A,B)<\eta$, which also implies $\|f-g\|_1<\eta$. We use now Lemma \ref{funcion f} for $f$ and $h=g$ which gives us $K>0$ (depending on $\eta$ and $A$). Let
\begin{equation}\label{fg2}
E_{f}=\{f\leq K\}\quad\text{ and }\quad E_{g}=\{g\leq K\}.
\end{equation}
Then Lemma \ref{funcion f} gives that
\begin{equation}\label{funcion f2}
\int_{E^{c}_{h}}h\:d\mu(\w)<2\eta\quad\text{and}\quad\mu(E^{c}_{h})<\frac{2\eta}{K},\quad\text{for}\quad h=f, g.
\end{equation}
Set
\begin{equation}\label{defG}
G=\bigcap\limits^{T-1}_{i=0}\varphi^{-i}(E_{f}\cap E_{g}).
\end{equation}
Then $G^{c}$ has \emph{small measure}: \begin{equation}\label{mu(G)}
\begin{split}
\mu(G^{c})&\leq\sum_{i=0}^{T}\mu(\varphi^{-i}(E^{}_{f}\cup E^{}_{g})^{c}) \leq T\mu(E^{}_{f}\cup E^{}_{g})^{c}\leq T\frac{4\eta}{K}<\frac{\varepsilon}{4K}.
\end{split}
\end{equation}

Elements belonging to the set $G$ defined in \eqref{defG} have nice estimates. Indeed we have:

\begin{lemma}\label{leminha}
If $\w\in G$, then for $C=A,B$ we have $\int_0^T\|C(\varphi^s(\w))\|\,ds\leq TK$ and $\|\Phi_C(T,\w)\|\leq \text{e}^{TK}$.
\end{lemma}
\begin{proof}
Once $\w\in G$ we have $\varphi^i(\w)\in E_{f}$ for all $i=0,\ldots,T$, and $$\int_0^T\|A(\varphi^s(\w))\|\,ds=\sum_{i=0}^{T-1}\int_0^{1}\|A(\varphi^{s+i}(\w))\|\,ds\leq TK.$$ By \eqref{eqIC3} we get
$\log^+\|\Phi_{A}(T,\omega)\|\leq\int_{0}^{T}\|B(\varphi^{s}(\omega))\|\,ds$ and we are done. The other case is similar.
\end{proof}

Next result gives us a `weak form' of continuous dependence of the solutions on its infinitesimal generator which will be enough for our purposes. The dependence will have an $L^p$ flavour in the sense that the estimate on the separation of the two solutions will be on a set, despite huge in measure, not exactly the whole $M$ and the reason we call `weak'.

\begin{lemma}\label{contdep2}
For all $1\leq p < \infty$,
$$\left(\int_{G}\|\Phi_A(1,\w)-\Phi_B(1,\w)\|^p\,d\mu(\w)\right)^{\frac{1}{p}}\leq \text{e}^{2Kp}\hat\sigma_p(A,B).$$
\end{lemma}
\begin{proof}
For $\w\in G$ we have
\begin{eqnarray*}
\|\Phi_{A}(1,\omega)-\Phi_{B}(1,\omega)\|&\leq&\int\limits_{0}^{1}\underbrace{\|A(\varphi^{s}(\omega)\|}_{\beta(s)}\,\|\Phi_{A}(s,\omega)-\Phi_{B}(s,\omega)\|ds\\
&+&\underbrace{\int\limits_{0}^{1}\left\|A(\varphi^{s}(\omega))-B(\varphi^{s}(\omega))\right\|\,\|\Phi_{B}(s,\omega)\|ds}_{\alpha(t)}.
\end{eqnarray*}
Taking $u(t)=\|\Phi_{A}(t,\omega)-\Phi_{B}(t,\omega)\|$,  we get $u(t)\leq \int_0^t \beta(s)u(s)\,ds+\alpha(t)$.
Using Gr\"onwall's inequality we get $u(t)\leq \alpha(t)\exp \left(\int_0^t \beta(s)\,ds\right)$,
which by Lemma~\ref{leminha} implies 
$$\|\Phi_{A}(1,\omega)-\Phi_{B}(1,\omega)\|\leq \int\limits_{0}^{1}\left\|A(\varphi^{s}(\omega))-B(\varphi^{s}(\omega))\right\|\,\|\Phi_B(s,\w)\|\,ds\,\text{e}^{K}.$$ Since by Lemma~\ref{leminha} we have $\|\Phi_B(s,\w)\|\leq \text{e}^{sK}$, using Jensen's inequality we get
\begin{eqnarray*}
\|\Phi_{A}(1,\omega)-\Phi_{B}(1,\omega)\|^p&\leq& \left(\int\limits_{0}^{1}\left\|A(\varphi^{s}(\omega))-B(\varphi^{s}(\omega))\right\|\,\|\Phi_B(s,\w)\|\,ds\,\text{e}^{K}\right)^p\\
&\leq& \int\limits_{0}^{1}\left\|A(\varphi^{s}(\omega))-B(\varphi^{s}(\omega))\right\|^p\,ds\,\text{e}^{2 K p}.
\end{eqnarray*}
Integrating in both sides and using Fubini and change of coordinates (recalling that $\varphi^s$ preserves $\mu$) we get:
\begin{eqnarray*}
\int_{G}\|\Phi_{A}(1,\omega)-\Phi_{B}(1,\omega)\|^p\,d\mu(\omega)&\leq& \text{e}^{2K p}\int_{G}\int\limits_{0}^{1}\left\|A(\varphi^{s}(\omega))-B(\varphi^{s}(\omega))\right\|^p\,ds\,\,d\mu(\omega)\\
&\leq& \text{e}^{2K p}\int\limits_{0}^{1} \int_{G}\left\|A(\varphi^{s}(\omega))-B(\varphi^{s}(\omega))\right\|^p\,\,d\mu(\omega)\,ds\\
&=& \text{e}^{2K p}\int\limits_{0}^{1} \int_{\varphi^s(G)}\left\|A(\omega)-B(\omega)\right\|^p\,\,d\mu(\omega)\,ds\\
&\leq&\text{e}^{2K p}\int_{M}\left\|A(\omega)-B(\omega)\right\|^p\,\,d\mu(\omega)\\
&=&\text{e}^{2Kp}\hat\sigma_p(A,B).
\end{eqnarray*}
\end{proof}

	\subsection{On the upper semi-continuity of the top Lyapunov exponent function}
	
In this entire subsection we do not assume that the flow $\varphi^t$ is ergodic. We define the function

	\begin{equation*}
\begin{array}{crcl}
\mathscr{L}\colon &(\mathcal{G}^1,\sigma_p) & \longrightarrow & \mathbb{R} \\& A & \longmapsto &  \int_{M}\lambda_{1}(A,\w)\, d\mu(\w).
\end{array}
\end{equation*}
Notice that under the ergodic hypothesis over the flow $\varphi^t$ we have $\mathscr{L}(A)=\lambda_1(A)$. In order to prove  that $\mathscr{L}$ is upper semi-continuous when $\mathcal{G}^1$ is endowed with the $\sigma_p$ metric defined in \S\,\,\,\ref{top}, we begin by given a preliminary result.

\begin{lemma}\label{cont solu}
For all $t\in\mathbb{R}$, $\w\in M$ and $A,B\in\mathcal{G}^1$, we have
		\begin{equation*}
		\log^{+}\|\Phi_B(t,\omega)\|\leq \log^{+}\|\Phi_A(t,\omega)\| + \|\Phi_B(t,\omega)-\Phi_A(t,\omega)\|.
		\end{equation*}
\end{lemma}
	
\begin{proof}
The proof follows straightforwardly from the triangular inequality	
\begin{align*}
\|\Phi_B(t,\w)\|\leq\|\Phi_A(t,\w)\|+\|\Phi_B(t,\w)-\Phi_A(t,\w)\|.
\end{align*}	
and the fact that $\log^{+}(x + y)\leq\log^{+}(x) + y$ for all $x, y\geq0$.
\end{proof}

In the next result the reader will find similarities with the arguments in \cite[Theorem 2]{AB}. Nevertheless, we present novelties namely (i) the topology is planted in the infinitesimal generator of the object which provide the Lyapunov exponent (ii) this point brings several continuity dependency issues to be treated using \S\,\,\,\ref{cdres} and (iii) as we will notice the arguments entail often the idiosyncrasies of the flows and so several adaptations are done accordingly.

\begin{proposition}\label{upper sc}
 For all $1\leq p<\infty$, the function $\mathscr{L}$ is upper semicontinuous when we endow $\mathcal{G}^1$ with the $\sigma_p$-topology, that is, for all $A\in\mathcal{G}^1$ and $\varepsilon>0$ there is $\delta>0$ such that $\sigma_p(A,B)<\delta$ implies $\mathscr{L}(B)<\mathscr{L}(A) + \varepsilon$ .
\end{proposition}

\begin{proof}
By Proposition~\ref{complete} we have $\sigma_1(A,B)\leq\sigma_p(A,B)$, for all $1\leq p<\infty$, thus it is enough to consider $p=1$. Let $A\in\mathcal{G}^1$ and $\varepsilon>0$ be given. We assume first that for $\mu$ almost every $\w\in M$ we have
	\begin{equation}\label{hatlambda geq 0}
\lambda_{1}(A,\w)\geq 0.
\end{equation}
From the subadditive ergodic theorem, we know that the limit
$$\lambda_1(A,\w)=\lim\limits_{t\rightarrow{+\infty}}\frac{1}{t}\log\|\Phi_A(t,\w)\|$$
holds almost everywhere and also in $L^1$. Hence, using \eqref{hatlambda geq 0}, we get $$\lim\limits_{t\rightarrow{+\infty}}\frac{1}{t}\int_M \log^{-}\|\Phi_A(t,\w)\|\,d\mu(\w)=0,$$
where $f^{-}:=\min\{f,0\}$.
Notice that
\begin{eqnarray}
\mathscr{L}(A)&=&\underset{t\rightarrow{+{\infty}}}{\lim}\:\frac{1}{t}\int_{M}\log\|\Phi_A(t,\omega)\|\,d\mu(\w)\nonumber=\underset{n\rightarrow{+{\infty}}}{\lim}\:\frac{1}{n}\int_{M}\log\|\Phi_A(n,\omega)\|\,d\mu(\w)\nonumber\\
&=&\underset{n\in\N}{\inf}\,\frac{1}{n}\int_{M}\log\|\Phi_A(n,\omega)\|\,d\mu(\w)\label{exp via inf}
\end{eqnarray}
	so that it is possible to find $T\in\N$ large enough in order to have
\begin{equation*}
\frac{1}{T}\int_{M}\log^-\|\Phi_A(T,\omega)\|\,d\mu(\w)>-\frac\varepsilon8\quad\text{and}\quad\frac{1}{T}\int_{M}\log\|\Phi_A(T,\omega)\|\,d\mu(\w)<\mathscr{L}(A)+\frac\varepsilon8
\end{equation*}
and therefore, since $f=f^- + f^+$,
\begin{equation}\label{lambda A}
\frac{1}{T}\int_{M}\log^+\|\Phi_A(T,\omega)\|\,d\mu(\w)\leq\mathscr{L}(A)+\frac\varepsilon4.
\end{equation}
We apply Lemma~\ref{funcion f} to $f$ as in~\eqref{fg} and  $\eta=\varepsilon/(16(T+2))$, giving us $K$ as in the statement. Set 
$\delta'=\min\left\{\eta,\eta T \text{e}^{-K(T+2)})\right\}$ and $\delta={\delta'}/{(1+\delta')}$. Fix $B\in \mathcal{G}^1$ satisfying $\sigma_1(A,B)<\delta$, which implies $\hat\sigma_1(A,B)< \delta'\leq\eta$. We use $K$, $T$, $\eta$ and $B$ to define the sets $E_f$, $E_g$ and $G$ as in \eqref{fg2} and \eqref{defG} respectively. We are going to bound the expression \[\frac{1}{T}\int_{M}\log^{+}\|\Phi_B(T,\w)\|d\mu(\w)=\text{(I)}+\text{(II)},\]
where
\[\text{(I)}=\frac{1}{T}\int_{G^c}\log^{+}\|\Phi_B(T,\w)\|d\mu(\w)\quad\text{ and }\quad \text{(II)}=\frac{1}{T}\int_{G}\log^{+}\|\Phi_B(T,\w)\|d\mu(\w).\]
 For the first part $\text{(I)}$, and by \eqref{eqIC3} we have
\begin{eqnarray*}
\frac{1}{T}\int_{G^c}\log^{+}\|\Phi_B(T,\w)\|\,d\mu(\w)&\leq& \frac{1}{T}\int_{G^c}\int_0^T\|B(\varphi^s(\w))\|\,ds\,d\mu(\w)\leq \frac{1}{T}\int_{G^c} \sum_{i=0}^{T-1} g(\varphi^i(\w))\,d\mu(\w)\\
&=& \frac{1}{T}\sum_{i=0}^{T-1} \int_{G^c}  g(\varphi^i(\w))\,d\mu(\w)= \frac{1}{T}\sum_{i=0}^{T-1} \int_{\varphi^i(G^c)}  g(\w)\,d\mu(\w).
\end{eqnarray*}
For each $i=0,\dots,T-1$ we have, by \eqref{funcion f2} and relation \eqref{mu(G)},
$$\int_{\varphi^{i}(G^{c})}\!\!g\: d\mu(\w)\leq\int_{E^{c}_{g}}\!\!g\: d\mu(\w) + \int_{E_{g}\cap\varphi^{i}(G^{c})}\!\!g\: d\mu(\w)<2\eta + K\mu(E_{g}\cap\varphi^{i}(G^{c}))\le\frac\varepsilon{8} + K\mu(G^{c})\leq\frac\varepsilon2.$$
Putting all together we get
\begin{equation}\label{six eta}
\text{(I)}=\frac{1}{T}\int_{G^c}\log^{+}\|\Phi_B(T,\w)\|d\mu(\w)\leq \frac\varepsilon2.
\end{equation}
Next we estimate the second part $\text{(II)}$. Using Lemma \ref{cont solu} and \eqref{lambda A} we have
\begin{align}
\text{(II)}&\displaystyle\leq\frac{1}{T}\int_{G}\log^+\|\Phi_A(T,\w)\|d\mu(\w) + \frac{1}{T}\int_{G}\|\Phi_B(T,\omega)-\Phi_A(T,\omega)\|d\mu(\w)\nonumber\\&\displaystyle\leq\mathscr{L}(A) + \frac\varepsilon4 + \frac{1}{T}\int_{G}\|\Phi_B(T,\omega)-\Phi_A(T,\omega)\|d\mu(\w).\label{LB}
\end{align}
To estimate the integral on the right hand side, we proceed as follows. Using the cocycle properties and Lemma~\ref{leminha} we have for all $\w\in G$ and  $i=1,\dots,T-1$:
\begin{align*}
\|\Phi_B(i+1,\omega)-\Phi_A(i+1,\omega)\|&\leq\|\Phi_B(1,\varphi^{i}(\omega))\Phi_B(i,\w) - \Phi_A(1,\varphi^{i}(\omega))\Phi_A(i,\w)\|\\
&\leq\|\Phi_B(1,\varphi^{i}(\w))\|\|\Phi_B(i,\omega)-\Phi_A(i,\omega)\| + \\
&\,\,\,+\|\Phi_B(1,\varphi^{i}(\w))-\Phi_A(1,\varphi^{i}(\w))\|\|\Phi_A(i,\w)\|\\&\leq \text{e}^{K}\|\Phi_B(i,\omega)-\Phi_A(i,\omega)\|+\text{e}^{Ki}\|\Phi_B(1,\varphi^{i}(\w))-\Phi_A(1,\varphi^{i}(\w))\|.
\end{align*}
Integrating over $G$ we get by Lemma~\ref{contdep2} with $p=1$ that
$$\int_{G}\|\Phi_B(i+1,\omega)-\Phi_A(i+1,\omega)\|d\mu(\w)\leq \text{e}^{K}\int_{G}\|\Phi_B(i,\omega)-\Phi_A(i,\omega)\|d\mu(\w)+ \text{e}^{Ki+2K}\delta'$$
By induction, we obtain for all $i=1,\dots,T$
$$\int_{G}\|\Phi_B(i,\omega)-\Phi_A(i,\omega)\|d\mu(\w)\leq (i+2)\text{e}^{K(i+2)}\delta'.$$
In particular if we take $i=T$ we get
\begin{equation}\label{LB1}
\int_{G}\|\Phi_B(T,\w)-\Phi_A(T,\w)\|d\mu(\w)\leq (T+2)\text{e}^{K(T+2)}\delta'\leq \frac\varepsilon4 T.
\end{equation}
Finally, from~\eqref{LB} and~\eqref{LB1} we get
\begin{equation}\label{LB2}
\text{(II)}=\frac{1}{T}\int_{G}\log^{+}\|\Phi_A(T,\w)\|d\mu(\w)\leq\mathscr{L}(A) + \frac\varepsilon2.
\end{equation}
To complete we consider \eqref{six eta} and \eqref{LB2} to conclude that
\begin{align*}
\mathscr{L}(B)&\leq\frac{1}{T}\int_{M}\log\|\Phi_B(T,\w)\|d\mu(\w)\\&=\frac{1}{T}\int_{G^c}\log^{+}\|\Phi_B(T,\w)\|d\mu(\w) + \frac{1}{T}\int_{G}\log^{+}\|\Phi_B(T,\w)\|d\mu(\w)\\&\leq\frac\varepsilon2+\left(\mathscr{L}(A) + \frac\varepsilon2\right)+ =\mathscr{L}(A) + \varepsilon.
\end{align*}
Let us prove now the general case. Again, let $A\in\mathcal{G}^1$ and $\varepsilon>0$ be given. For $\alpha>0$ we define the $\varphi^{t}$-invariant set 
$$L_\alpha=\{\omega\in M: \lambda_{1}(A,\w)<-\alpha\}.$$ Consider $\alpha$ large enough such that
		
\begin{equation}\label{gen case 1}
\int_{L_\alpha}\log^+\|\Phi_A(1,\omega)\|\,d\mu(\w) <\frac\varepsilon8\,\,\,\, \text{and} \,\,\,\, \int_{L_\alpha}\lambda_{1}(A,\w)\,d\mu(\w)>-\frac\varepsilon4.
\end{equation}
In particular we get
\begin{equation}\label{gen case trocas}
\int_{L_\alpha^c}\lambda_{1}(A,\w)\,d\mu(\w)<\mathscr{L}(A)+\frac\varepsilon4.
\end{equation}

Consider the (non kinetic) infinitesimal generator defined by $A+\alpha Id$. We claim that the weak solution of  \eqref{eq:LDS} considering the generator $A+\alpha Id$ is given by $\Phi_{A+\alpha Id}=\text{e}^{\alpha t}\Phi_A(t,\w)$. Indeed, it suffices to prove that 
\begin{equation}\label{dddd}
\text{e}^{\alpha t}\Phi_A(t,\w)=\text{Id}+\int_{0}^{t}(A+\alpha Id)(\varphi^{s}(\omega))\text{e}^{\alpha s}\Phi_A(s,\w)ds
\end{equation}
that is
$$\text{e}^{\alpha t}\Phi_A(t,\w)=\text{Id}+\int_{0}^{t}\text{e}^{\alpha s}A(\varphi^{s}(\omega))\Phi_A(s,\w)ds+\alpha\int_{0}^{t}\text{e}^{\alpha s}\Phi_A(s,\w)ds.$$ 
Using integrating by parts $u=\text{e}^{\alpha s}$, $dv=A(\varphi^{s}(\omega))\Phi_A(s,\w)\,ds$, $du=\alpha \text{e}^{\alpha t}\,ds$ and $v=\Phi_A(s,\w)$ we get 
$$\text{e}^{\alpha t}\Phi_A(t,\w)=\text{Id}+\text{e}^{\alpha t}\Phi_A(t,\w)-\text{Id}-\int_{0}^{t}\alpha \text{e}^{\alpha s}\Phi_A(s,\w)\,ds+\alpha\int_{0}^{t}\text{e}^{\alpha s}\Phi_A(s,\w)\,ds,$$
and the claim \eqref{dddd} proved. Now, we define its `maximal Lyapunov exponent' by 
$$\lambda_{1}(A+\alpha Id,\w):=\lim\limits_{t\to\infty}\frac{1}{t}\log^+\|\text{e}^{\alpha t}\Phi_A(t,\omega)\|.$$
Notice that $\lambda_{1}(A+\alpha Id,\w)=\lambda_{1}(A,\w)+\alpha$ and if $\w\in L^{c}_{\alpha}$ we have $\lambda_{1}(A+\alpha Id,\w)\geq0$.
Define now

\begin{equation*}
\begin{array}{rrcl}
\tilde{\mathscr{L}}\colon &(\mathcal{G}^1\,,\sigma_p)&\longrightarrow & \mathbb{R} \\& A & \longmapsto &  \int_{M}\lambda_{1}(A+\alpha Id,\w)d\mu(\w).
\end{array}
\end{equation*}
Proceeding similarly to the previous computations for $\mathscr{L}$, we have that $\tilde{\mathscr{L}}$ is upper semicontinuous if we restrict $A+\alpha Id$ to $L_\alpha^c$. For this, notice also that Lemma~\ref{contdep2} also holds if we replace $\Phi_A$ and $\Phi_B$ by the corresponding $\text{e}^\alpha\Phi_A$ and $\text{e}^{\alpha}\Phi_B$. Hence there is $\delta>0$ such that if $\sigma_p((A+\alpha\text{Id})\vert_{L_\alpha^c}, (B+\alpha\text{Id})\vert_{L_\alpha^c})<\delta$ we have
\begin{align*}
\int_{L_\alpha^c}\lambda_{1}(B+\alpha Id,\w)d\mu(\w)&\leq\int_{L_\alpha^c}\lambda_{1}(A+\alpha Id,\w)d\mu(\w)+\frac\varepsilon4,
\end{align*}
that is,
\begin{equation}\label{gen case 2}
\int_{L_\alpha^c}\lambda_{1}(B,\w)d\mu(\w)\leq\int_{L_\alpha^c}\lambda_{1}(A,\w)d\mu(\w)+\frac\varepsilon4.
\end{equation}

On the other hand, since $L_{\alpha}$ is invariant,  we have
\begin{equation}\label{lbd log+}
\int_{L_a}\lambda_{1}(\Phi_B,\w)\,d\mu(\w)= \inf_n\frac1n\int_{L_a}\log^+\|\Phi_B(n,\omega)\|\,d\mu(\w)
\leq\int_{L_a}\log^+\|\Phi_B(1,\omega)\|\,d\mu(\w)
\end{equation}
Consider $T=1$ and $\eta$, $K$ and $G$  as before, and replace $\eta$ by $\eta'=\eta/4$.  From Lemma \ref{cont solu},~\eqref{gen case 1} and \eqref{LB1} we get
\begin{eqnarray}
\int_{L_\alpha\cap G} \log^+\|\Phi_B(1,\omega)\|\,d\mu(\w) &\leq&\displaystyle  \int_{L_\alpha\cap G} \log^+\|\Phi_A(1,\omega)\|\,d\mu(\w)+\int_{L_\alpha\cap G} \|\Phi_B(1,\omega)-\Phi_A(1,\omega)\|\,d\mu(\w)\nonumber
\\
 &\leq&  \frac\varepsilon{16}+\frac\varepsilon{16}\nonumber\\
 &=&\frac\varepsilon8.\label{laG}
\end{eqnarray}
Similarly as to~\eqref{six eta}, we obtain
\begin{eqnarray*}
\int_{L_\alpha\cap G^c} \log^+\|\Phi_B(1,\omega)\|\,d\mu(\w) &\leq \frac\varepsilon8,
 \end{eqnarray*}
which together with~\eqref{laG} leads to
 \begin{eqnarray}\label{laB}
\int_{L_\alpha} \log^+\|\Phi_B(1,\omega)\|\,d\mu(\w) &\leq \frac\varepsilon4,
 \end{eqnarray}
 
Finally, from~\eqref{gen case 2},~\eqref{lbd log+},~\eqref{gen case trocas}  and~\eqref{laB} we have
\begin{align*}
\mathscr{L}(B)&=\int_{L_\alpha}\lambda_1(B,\w)d\mu(\w)+\int_{L_\alpha^c}\lambda_1(B,\w)d\mu(\w)
\\
&\leq\int_{L_\alpha}\lambda_1(A,\w)d\mu(\w)+\frac\varepsilon4+\int_{L_\alpha}\log^+\|\Phi_B(T,\w)\|d\mu(\w)
\\&\leq\left(\mathscr{L}(A) + \frac\varepsilon2\right)+ \frac\varepsilon4 + \frac\varepsilon4=\mathscr{L}(A) + \varepsilon.
\end{align*} 

\end{proof}
	
\medskip
		
	\section{Proof of Theorem~\ref{ops2}}\label{PT1}

	The strategy applied in $C^0$ cocycles endowed with the $C^0$ norm (or essential bounded cocycles endowed with the $L^\infty$ norm) developed in \cite{B, BV2, Be, Be2} to diminish Lyapunov exponents cannot be used in our context. Indeed, $L^p$ norms catch information on a neighborhood of a segment of orbit and, contrary to the $C^0$ norm, not from the segment itself. For this reason we must follow a different approach.
	
We recall that since we are assuming that $\varphi^t$ is ergodic, the Lyapunov exponents of a given
$A\in\mathcal K^1$   are constant $\mu$ almost everywhere and referred as $\lambda_1(A)\geq\lambda_2(A)$. We define the \emph{jump map} by:

\begin{equation*}
 \begin{array}{crcl}
\mathscr{J}\colon &(\mathcal{K}^1\,,\sigma_p) & \longrightarrow & \mathbb{R} \\& A & \longmapsto &\frac{\lambda_{1}(A)-\lambda_{2}(A)}{2}.
\end{array}
\end{equation*}
	
\bigskip

Next result is a straightforward consequence of Proposition~\ref{MProp} and is crucial to obtain the proof of Theorem~\ref{ops2}.
	
\begin{lemma}\label{basics}
Consider $1\leq p<\infty$ and let $A\in\mathcal{K}^1$ and $\varepsilon,\delta>0$ be given. There exists $B\in \mathcal{K}^1$ with $\sigma_p(A,B)<\varepsilon$ such that
 \begin{equation*}
\mathscr{L}(B)<\delta-\mathscr{J}(A)+\mathscr{L}(A).
 \end{equation*}
\end{lemma}

\begin{proof}
From (i) in Proposition~\ref{complete} and Proposition~\ref{MProp} there is $B\in \mathcal{K}^1$ with $\sigma_1(A,B)\leq \sigma_p(A,B)<\varepsilon$ such that
 \begin{equation*}
 \mathscr{L}(B)<\frac{\lambda_{1}(A)+\lambda_{2}(A)}{2}+\delta=\delta-\mathscr{J}(A)+\mathscr{L}(A).
 \end{equation*}
\end{proof}
	
\begin{theorem}\label{contiunity points}
 Consider $1\leq p<\infty$ and the complete metric space $(\mathcal{K}^1,\sigma_p)$. If $A\in\mathcal{K}^1$ is a continuity point of $\mathscr{L}$, then $\mathscr{J}(A)=0$.
\end{theorem}
	
\begin{proof}
 We take a sequence of kinetic linear differential systems $A_n\in\mathcal{K}^1$ converging to $A\in\mathcal{K}^1$ in the $\sigma_p$-sense. Since $A$ is a continuity point we must have $\lim\mathscr{L}(A_n)=\mathscr{L}(A)$.
 By Lemma~\ref{basics}, given $\varepsilon_n\rightarrow 0$ and $\delta>0$, there exists $B_n\in\mathcal{K}^1$, with $\sigma_p(A_n,B_n)<\varepsilon_n$, such that
		
$$\mathscr{L}(B_n)<\delta-\mathscr{J}(A_n)+\mathscr{L}(A_n).$$
Considering limits on $n$ we get
		
$$\underset{n\rightarrow\infty}{\lim}\mathscr{L}(B_n)<\delta-\underset{n\rightarrow\infty}{\lim}\mathscr{J}(A_n)+\mathscr{L}(A).$$
		Since $A$ is a continuity point of $\mathscr{L}$ we obtain that $\mathscr{J}(A_n)=0$ for all $n$ sufficiently large and thus $\mathscr{J}(A)=0$.
\end{proof}

We are now in codition to finish the proof of our main result.	
	\begin{proof}(of Theorem~\ref{ops2})
By Proposition~\ref{upper sc} the function $\mathscr{L}$ is upper semicontinuous and by Theorem~\ref{contiunity points} the continuity points of $\mathscr{L}$ have trivial spectrum (\emph{jump} equal to zero). It is well-known that the set of points of continuity of an upper semicontinuous function is a residual subset (see \cite{K}). Thus, there exists an $\sigma_p$-residual subset $\mathcal{R}\subset\mathcal{K}^1$ such that if $A\subset\mathcal{R}$, then $\mathscr{J}(A)=0$, that is $\lambda_1(A)=\lambda_2(A)$.
\end{proof}

From Corollary \ref{complete 2} we have that $(\mathcal{K}^1,\sigma_p)$ is a Baire space, hence a $\sigma_p$-residual is $\sigma_p$-dense. Therefore, we obtain a $\sigma_p$-prevalence of trivial Lyapunov spectrum among kinetic systems.
	
\vspace{1cm}

\textbf{Acknowledgements:} The authors were partially supported by FCT - `Funda\c{c}\~ao para a Ci\^encia e a Tecnologia', through Centro de Matem\'atica e Aplica\c{c}\~oes (CMA-UBI), Universidade da Beira Interior, project UID/00212/2020. MB also like to thank CMUP for providing the necessary conditions in which this work was developed.


\bigskip



\vspace{0.4cm}

\begin{thebibliography}{00}
		
	
\bibitem{AB} A. Arbieto, J. Bochi, \emph{$L^p$-generic cocycles have one-point Lyapunov spectrum}, Stochastics and Dynamics 3 (2003) 73--81. Corrigendum. ibid, 3 (2003) 419--420.
		
\bibitem{A} L. Arnold, \emph{Random Dynamical Systems}, Springer Verlag, 1998.





		
\bibitem{ACE} L. Arnold, H. Crauel, J.-P. Eckmann, editors \emph{Lyapunov Exponents.} Proceedings, Oberwolfach 1990, volume 1486 of Springer Lecture Notes in Math. Springer-Verlag, Berlin Heidelberg New York, 1991.
		
\bibitem{AW} L. Arnold, V. Wihstutz, editors, \emph{Lyapunov Exponents.} Proceedings, Bremen 1984, volume 1186 of Springer Lecture Notes in Mathematics. SpringerVerlag, Berlin Heidelberg New York, 1986.

		\bibitem{AABT2} A. Azevedo, D. Azevedo, M. Bessa and  M. J. Torres,
	\newblock \textit{The closing lemma and the planar general density theorem for Sobolev maps}, Proc. Amer. Math. Soc.,149 (2021), 1687--1696.
		
		

\bibitem{Ba} L. Barreira, \emph{Lyapunov Exponents}, Birkhauser Verlag, 2017.
		

		
	
\bibitem{Be} M. Bessa, \emph{Dynamics of generic 2-dimensional linear differential systems}, J. Differential Equations 228 (2) (2006) 685--706.
		
	
\bibitem{Be2} M. Bessa, \emph{Perturbations of Mathieu equations with parametric excitation of large period} Advances in Dynamical Systems and Applications, 7, 1, (2012) 17--30.
		

\bibitem{BVi} M. Bessa, H. Vilarinho, \emph{Fine properties of $L^p$-cocycles which allows abundance of simple and trivial spectrum}. J. Differential Equations, 256, 7 (2014) 2337--2367.

\bibitem{B} J. Bochi, \emph{Genericity of zero Lyapunov exponents}, Ergod. Th. \& Dynam. Sys. 22 (2002) 1667--1696.
		
\bibitem{BV2} J. Bochi, M. Viana, \emph{The Lyapunov exponents of generic volume-preserving and symplectic maps}, Ann. of Math. 161 (3) (2005) 1423--1485.

		

		
\bibitem{DK} P. Duarte, S. Klein, \emph{Lyapunov exponents of linear cocycles. Continuity via large deviations.} Atlantis Studies in Dynamical Systems, 3. Atlantis Press, Paris, 2016.	
		
		\bibitem{FHT0} Edson de Faria, Peter Hazard, Charles Tresser,
	\newblock \textit{Infinite entropy is generic in H\"older and Sobolev spaces},
	\newblock C. R. Acad. Sci. Paris, Ser. I,  355(11) (2017), 1185--1189.
	\bibitem{FHT} E. de Faria, P. Hazard and C. Tresser,
	\newblock \textit{Genericity of infinite entropy for maps with low regularity},
	\newblock Ann. Sc. Norm. Super. Pisa Cl. Sci. (5) XXII (2021), 601--664.

		

		
\bibitem{FL} X. Feng, K. Loparo, \emph{Almost sure instability of the random harmonic oscillator}, SIAM J. Appl. Math. 50, 3, (1990) 744--759.
		

\bibitem{Iz} N. A. Izobov, \emph{Lyapunov exponents and stability. Stability, Oscillations and Optimization of Systems}, 6. Cambridge Scientific Publishers, Cambridge, 2012.
		
\bibitem{JPS} R. Johnson, K. Palmer, G. Sell, \emph{Ergodic properties of linear dynamical systems}, SIAM J. Math. Anal. 18 (1987) 1--33.
		

		
\bibitem{K} K. Kuratowski, \emph{Topology, vol. 1}, Academic Press, 1966.
		
\bibitem{Lei} A. Leizarowitz, \emph{On the Lyapunov exponent of a harmonic oscillator driven by a finite-state Markov process}, SIAM J. Appl. Math., 49, 2, (1989) 404--419.
		

		
\bibitem{M1} R. Ma\~{n}\'{e}, \emph{Oseledec's theorem from generic viewpoint}, Proceedings of the international Congress of Mathematicians, Warszawa vol. 2 (1983) 1259--1276.
		
\bibitem{M2} R. Ma\~{n}\'{e}, \emph{The Lyapunov exponents of generic area preserving diffeomorphisms}, International Conference on Dynamical Systems (Montevideo, 1995), Pitman Res. Notes Math. Ser. 362 (1996) 110--119.
		
\bibitem{Mi} V. M. Millionshchikov, \emph{Systems with integral separateness which are dense in the set of all linear systems of differential equations},  Differential Equations 5 (1969) 850--852.
		

		
\bibitem{No}  V. L. Novikov, \emph{Almost reducible systems with almost periodic coefficients}, Mat. Zametki 16 (1974) 789--799.
		
\bibitem{O} V. Oseledets, \emph{A multiplicative ergodic theorem: Lyapunov characteristic numbers for dynamical systems}, Transl. Moscow Math. Soc. 19 (1968) 197-231.

\bibitem{Felix}	F. Otto, \emph{$L^1$-contraction and uniqueness for quasilinear elliptic-parabolic equations},  J. Differential Equations 131 (1996) 20--38. 


\bibitem{Pi} A. Pikovsky, A. Politi, \emph{Lyapunov exponents. A tool to explore complex dynamics.} Cambridge University Press, Cambridge, 2016.
		
\bibitem{Rudo}	D. Rudolph,	\emph{A Two-Valued Step Coding for Ergodie Flows}, Math. Z. 150 (1976) 201--220.

\bibitem{V} M. Viana, \emph{Lectures on Lyapunov exponents}. Cambridge Studies in Advanced Mathematics, 145. Cambridge University Press, Cambridge, 2014.

\end{thebibliography}
\end{document}